\documentclass[reqno]{amsart}
\usepackage{amsmath,amssymb,amsfonts}
\usepackage{euscript}
\usepackage{enumerate}
\usepackage{verbatim}
\newtheorem{theorem}{Theorem}
\newtheorem{lemma}{Lemma}
\newtheorem{proposition}{Proposition}
\newtheorem{example}{Example}

\renewcommand{\epsilon}{\varepsilon}
\newcommand{\eps}{\varepsilon}

\newcommand{\dr}{\, dr}
\DeclareMathOperator{\Id}{Id}

\newcommand{\prts}[1]{\left(#1\right)}
\newcommand{\prtsr}[1]{\left[#1\right]}
\newcommand{\pfrac}[2]{\prtsr{\dfrac{#1}{#2}}}
\newcommand{\abs}[1]{\left|#1\right|}
\newcommand{\set}[1]{\left\{#1\right\}}
\newcommand{\norm}[1]{\left\|#1\right\|}
\usepackage{dsfont}
   \newcommand{\N}{\ensuremath{\mathds N}}
   \newcommand{\R}{\ensuremath{\mathds R}}

\def\cB{\EuScript{B}}

\def\cV{\EuScript{V}}
\def\cX{\EuScript{X}}
\usepackage{pstricks}
\begin{document}
\title[Nonautonomous equations, generalized dichotomies ...]
   {Nonautonomous equations, generalized dichotomies and stable manifolds}
\author{Ant\'onio J. G. Bento}
\address{A. J. G. Bento\\
   Departamento de Matem\'atica\\
   Universidade da Beira Interior\\
   6201-001 Covilh\~a\\
   Portugal}
\email{bento@mat.ubi.pt}
\author{C\'esar Silva}
\address{C. Silva\\
   Departamento de Matem\'atica\\
   Universidade da Beira Interior\\
   6201-001 Covilh\~a\\
   Portugal}
\email{csilva@mat.ubi.pt}
\urladdr{www.mat.ubi.pt/~csilva}
\date{\today}
\begin{abstract}
   Assuming the existence of a general nonuniform dichotomy for the evolution
   operator of a non-autonomous ordinary linear differential equation in a
   Banach space, we establish the existence of invariant stable manifolds for
   the semiflow generated by sufficiently small nonlinear perturbations of the
   linear equation. The family of dichotomies considered satisfies a general
   growth rate given by some increasing differentiable function, allows
   situations for which the classical Lyapunov exponents are zero, and contains
   the nonuniform exponential dichotomies as a very particular case. In
   addition we also give explicit examples of linear equations that admit all
   the possible considered dichotomies.
\end{abstract}
\maketitle
\section{Introduction}
The existence of invariant manifolds is an important subject in the theory of
dynamical systems and differential equations. The fundamental tools to
establish the existence of invariant manifolds are, for dynamical systems, the
notion of nonuniform hiperbolicity and, for differential equations, the related
notion of nonuniform dichotomy.

The concept of nonuniform hyperbolicity, introduced by
Pesin~\cite{Pesin-IANSSSR-1976,Pesin-UMN-1977,Pesin-IANSSSR-1977}, is a
generalization of the classical concept of (uniform) hyperbolicity where the
rates of expansion and contraction are allowed to vary from point to point. For
nonuniformly hyperbolic trajectories, Pesin~\cite{Pesin-IANSSSR-1976} was able
to obtain a stable manifold theorem in the finite-dimensional setting. Since
then, there were several contributions to the theory. Namely, in
\cite{Ruelle-IHESPM-1979} Ruelle gave a proof of the stable manifold theorem
based on the study of perturbations of products of matrices occurring in
Oseledets' multiplicative ergodic theorem in~\cite{Oseledets-TMMS-1968}.
Another proof, using graph transform techniques and based on the classical work
of Hadamard, was given by Pugh and Shub in~\cite{Pugh-Shub-TAMS-1989}.
Following his approach in \cite{Ruelle-IHESPM-1979}, Ruelle
in~\cite{Ruelle-AM-1982} proved a stable manifold theorem in Hilbert spaces
under some compactness assumptions. In~\cite{Mane-LNM-1983} Ma\~n\'e obtained a
corresponding version for transformations in Banach spaces under some
compactness and invertibility assumptions, that includes the case of
differentiable maps with compact derivative at each point, and, in
\cite{Thieullen-AIHPAN-1987}, Thieullen  generalized the results of Ma\~n\'e
for a family of transformations satisfying some asymptotic compactness.

On the other hand, in the setting of nonautonomous differential equations,
Barreira and Valls~\cite{Barreira-Valls-JDE-2006,Barreira-Valls-CMP-2005}
introduced the notion of nonuniform exponential dichotomy based on the
classical notion of exponential dichotomy  introduced by Perron
in~\cite{Perron-MZ-1930} and also in the notion of nonuniformly hyperbolic
trajectory introduced by Pesin
in~\cite{Pesin-IANSSSR-1976,Pesin-UMN-1977,Pesin-IANSSSR-1977}. Versions of the
stable manifold theorems for nonuniformly exponential dichotomies were also
obtained, both in the continuous and the discrete time settings. In fact, for
flows and semiflows arising from nonautonomous ordinary differential equations,
Barreira and Valls were able to obtain stable manifold theorems in several
contexts. For more details about the stability theory of nonautononous
differential equations with nonuniform exponential dichotomies and, in
particular, the existence of invariant manifolds, the reader can consult the
book~\cite{Barreira-Valls-LNM-2008}. Corresponding results were also obtained
in the discrete time setting, namely  in~\cite{Barreira-Valls-DCDS-2006},
Barreira and Valls obtained $C^1$ stable manifolds for nonuniformly exponential
dichotomies in finite dimension and, using this result as a starting point,
in~\cite{Barreira-Silva-Valls-JDDE-2008} it was established the existence of
$C^k$ local manifolds for $C^k$ perturbations by an induction process that uses
a linear extension of the dynamics. For Banach spaces, assuming a nonuniform
exponential dichotomy, it was established
in~\cite{Barreira-Silva-Valls-JLMS-2008} the existence of $C^1$ global stable
manifolds for some perturbations of linear dynamics.

The purpose of this paper is to obtain global stable manifolds for
perturbations of linear ordinary differential equations, assuming some general
type of dichotomy for the evolution operator associated with the linear
equation. This dichotomies bound the norms of the evolution operator by a
nonuniform law that is not necessarily exponential. In fact, the dichotomies
considered (see~\eqref{eq:dich-1} and~\eqref{eq:dich-2}) contain the nonuniform
exponential dichotomies as a very particular case. The existence of global
stable manifolds for perturbations of linear ordinary differential equations
with nonuniform dichotomies that are not exponential was already addressed
in~\cite{Bento-Silva-JDE} for the particular case of polynomial dichotomies
(more precisely, in that paper the dichotomies considered correspond to the
dichotomies obtained in the present case by setting $\mu(t)=t+1$
in~\eqref{eq:dich-1} and~\eqref{eq:dich-2}). In the discrete time setting, this
problem was discussed in~\cite{Bento-Silva-JFA-2009} for perturbations of some
nonuniform polynomial dichotomies.

The perturbations considered here include as a particular case the ones
considered in~\cite{Barreira-Valls-DCDS-2008} and~\cite{Bento-Silva-JDE}.
Therefore, the theorems proved there, respectively for nonuniform exponential
dichotomies and for nonuniform polynomial dichotomies, are particular cases of
the result obtained in this paper. We emphasize that, contrarily to what
happens in the exponential case, we need to prove
inequality~\eqref{eq:norm_x_phi-x_psi} without using Gronwall's Lemma. This was
done in Lemma~\ref{lemma:without-Gronwall} by mathematical induction.

Notice that the family of dichotomies considered in this paper includes cases
where the Lyapunov exponent considered in \cite{Barreira-Valls-LNM-2008} for
Hilbert spaces is zero for all $v \in E_1$ (see Section \ref{section:NP} for
the definition of $E_1$). In fact, we only have finite nonzero Lyapunov
exponent if the growth rate is, is some sense, close to exponential.

On the other hand, it is strait-forward to see that for a linear equation
$u'=A(t)u$ and initial conditions $u_0$ the map $\chi:X \to [-\infty,+\infty]$
given by
   $$ \chi(u_0)=\limsup_{n \to +\infty} \frac{\log \|u(t)\|}{\log \mu(t)}$$
is a Lyapunov exponent in the sense of the abstract theory developed
in~\cite{Bylov-Vinograd-Grobman-Nemyckii-1966}.
In~\cite{Barreira-Valls-DCDS-2008-1} Barreira and Valls used these Lyapunov
exponents to establish the existence of generalized trichotomies for linear
equations in finite dimension, assuming that the matrices $A(t)$ are in block
form,  and also that
   $$ \lim_{t \to +\infty} \frac{\log t}{\log \mu(t)}=0,$$
where the increasing functions $\mu$ are the growth rates considered there.

In another direction, in a recent work, Barreira and
Valls~\cite{Barreira-Valls-JFA-2009} established the persistence of generalized
dichotomies under small linear perturbations.  The dichotomies they consider
correspond to the dichotomies considered here though the definition is given in
slightly different manner, namely our definition correspond to the one given
there setting $\mu(t)=e^{\rho(t)}$ (compare~\eqref{eq:dich-1}
and~\eqref{eq:dich-2} with Section 3 of~\cite{Barreira-Valls-JFA-2009}).

The content of the paper is as follows: in section~\ref{section:NP}
we establish the setting and define the dichotomies to be considered;
then, in section~\ref{section:E}, we present, for each growth rate
$\mu$ in our family of growth rates, examples of nonuniform
$\mu$-dichotomies that are not uniform $\mu$-dichotomies;
finally in section~\ref{section:global} we prove our result on the existence
of
stable manifolds for a large family of sufficiently small perturbations of
the linear differential equations considered.

\section{Notation and Preliminaries}\label{section:NP}
Let $B(X)$ be the space of bounded linear operators acting on a Banach
space~$X$. We are going to consider the initial value problem
\begin{equation} \label{eq:ivp-li}
   v' = A(t) v, \ v(s) = v_s
\end{equation}
with $s \ge 0$ and $v_s \in X$ and where $A \colon \R^+_0 \to B(X)$ is a $C^1$
function. We assume that each solution of~\eqref{eq:ivp-li} is global and we
denote by $T(t,s)$ the evolution operator associated with~\eqref{eq:ivp-li},
i.e., $v(t) = T(t,s)v_s$ for $t \ge 0$.

Let $\mu : \R^+_0 \to [1, +\infty[$ be an increasing differentiable function
such that
   $$ \lim_{t \to + \infty} \mu(t) = +\infty.$$

We say that equation \eqref{eq:ivp-li} admits a \textit{nonuniform
$\mu$-dichotomy} in $\R^+_0$ if, for each $t \ge 0$, there are projections
$P(t)$
such that
\begin{equation*}
   P(t)T(t,s) = T(t,s) P(s), \ t, s \ge 0
\end{equation*}
and constants $D \ge 1$, $a < 0 \le b$ and $\eps \ge 0$ such that, for every
$t
\ge s \ge 0$,
\begin{align}
   & \| T(t,s)P(s)\|
      \le D \pfrac{\mu(t)}{\mu(s)}^a \mu(s)^\eps,\label{eq:dich-1}\\
   & \|T(t,s)^{-1} Q(t)\|
      \le D \pfrac{\mu(t)}{\mu(s)}^{-b} \mu(t)^\eps, \label{eq:dich-2}
\end{align}
where $Q(t)=\Id-P(t)$ is the complementary projection. When $\eps = 0$ we say
that we have a \textit{uniform $\mu$-dichotomy} or simply a
\textit{$\mu$-dichotomy}. We define for each $t \ge 0$ the linear subspaces
   $$ E(t)=P(t)X \quad \text{ and } \quad F(t)=Q(t)X.$$
Without loss of generality, we always identify the spaces $E(t) \times F(t)$
and $E(t) \oplus F(t)$ as the same space and in these spaces we use the
norm given by
   $$ \|(x,y)\| = \|x\|+\|y\|, \ \ (x,y)\in E(t) \times F(t).$$
Hence, the unique solution of~\eqref{eq:ivp-li} can be written in the form
   $$ v(t) = \prts{U(t,s) \xi, V(t,s) \eta}, \ t \ge s$$
where $v_s = \prts{\xi, \eta} \in E(s) \times F(s)$ and
   $$ U(t,s) := P(t) T(t,s) P(s) \quad \text{ and } \quad
      V(t,s) := Q(t) T(t,s) Q(s).$$
\section{Examples}\label{section:E}

In the following family of examples we are going to present nonautonomous
linear equations that admits a nonuniform $\mu$-dichotomy for each possible
function $\mu$.

\begin{example}
   Given $\eps > 0$ and $a < 0 \le b$, consider the differential equation in
   $\R^2$ given by
   \begin{equation} \label{eq:example}
      \begin{split}
         & u' = \prts{a \dfrac{\mu'(t)}{\mu(t)}
            + \omega \dfrac{\mu'(t)}{\mu(t)} (\cos t -1)
            - \omega \log \mu(t) \sin t} u\\
         & v' = \prts{b \dfrac{\mu(t)}{\mu'(t)}
            - \omega \dfrac{\mu'(t)}{\mu(t)} (\cos t -1)
            + \omega \log \mu(t) \sin t} v
      \end{split}
   \end{equation}
   where $\omega = \eps/2$. The evolution operator associated with this
   equation is given by
      $$ T(t,s)(u,v) = (U(t,s)u, V(t,s)v),$$
   where
   \begin{align*}
      & U(t,s) = \pfrac{\mu(t)}{\mu(s)}^a
         e^{\omega \log \mu(t) (\cos t -1) - \omega \log \mu(s) (\cos s
         -1)},\\
      & V(t,s) = \pfrac{\mu(t)}{\mu(s)}^b
         e^{-\omega \log  \mu(t) (\cos t -1) + \omega \log \mu(s) (\cos s
         -1)}.
   \end{align*}
   Let $P(t) \colon \R^2 \to \R^2$ be the projections defined by $P(t)(u,v)
   =(u,0)$ and $Q(t) = \Id - P(t)$. Then we have
   \begin{align*}
      & \| T(t,s)P(s)\|
         = \abs{U(t,s)}
         \le \pfrac{\mu(t)}{\mu(s)}^a \mu(s)^\eps\\
      & \|T(t,s)^{-1} Q(t)\|
         = \abs{V(t,s)^{-1}}
         \le \pfrac{\mu(t)}{\mu(s)}^{-b} \mu(t)^\eps
   \end{align*}
   and this shows that equation~\eqref{eq:example} admits a nonuniform
   $\mu$-dichotomy.

   Furthermore, since
      $$ U(2 k \pi,2k \pi - \pi)
         = \pfrac{\mu(2 k \pi)}{\mu(2k \pi -\pi)}^a
            \prts{\mu(2k \pi -\pi)}^\eps, \ \  k \in \N,$$
   the nonuniform part can not be removed.
\end{example}

The examples in~\cite{Bento-Silva-JDE} are the particular case of this
   family of examples obtained by setting $\mu(t)=t+1$.

\section{Stable manifolds} \label{section:global}
The purpose of this paper is the obtention of stable manifolds for the
nonlinear problem
\begin{equation} \label{eq:ivp-nonli}
   v' = A(t) v + f(t,v), \ v(s) = v_s
\end{equation}
when equation~\eqref{eq:ivp-li} admits a nonuniform $\mu$-dichotomy and $f
\colon \R^+_0 \times X \to X$ is a perturbation of class $C^1$ and there exists
$\delta>0$ such that, for every $t \ge 0$ and $u,v\in X$,
\begin{align}
   & f(t,0)=0,\quad \partial f(t,0)=0, \label{cond-f-0}\\
   & \|\partial f(t,u)\|
      \le \delta \mu'(t) \mu(t)^{-3 \eps - 1}, \label{cond-f-1}\\
   & \|\partial f(t,u)- \partial f(t,v)\|
      \le \delta \mu'(t) \mu(t)^{-3 \eps - 1} \|u-v\|, \label{cond-f-2}
\end{align}
where, for a question of simplicity, $\partial$ denotes the partial derivative
with respect to the second variable and $\mu(t)$ and $\eps$ are the same as
in~\eqref{eq:dich-1} and~\eqref{eq:dich-2}. A trivial application of the mean
value theorem combined with~\eqref{cond-f-1} yields
\begin{equation} \label{cond-f-3}
   \|f(t,u) - f(t,v)\| \le \delta \mu'(t) \mu(t)^{-3\eps-1} \|u - v\|
\end{equation}
for every $u, v \in X$ and, with $v= 0$, equation~\eqref{cond-f-3} becames
\begin{equation} \label{cond-f-4}
   \|f(t,u)\| \le \delta \mu'(t) \mu(t)^{-3\eps-1} \|u\|.
\end{equation}

For
\begin{equation} \label{def:G}
   G = \bigcup_{t \ge 0} \set{t} \times E(t)
\end{equation}
we define the space $\cX$ of $C^1$ functions $\phi \colon G \to X$
such that
\begin{align}
   & \phi(s,\xi) \in F(s), \label{cond-phi-in-F(s)}\\
   & \phi(s,0)=0,\quad \partial \phi(s,0)=0, \label{cond-phi-0}\\
   & \|\partial\phi(s,\xi)\| \le 1, \label{cond-phi-1}\\
   & \|\partial \phi(s,\xi) - \partial \phi(s,\bar\xi)\|
      \le \|\xi-\bar\xi\|, \label{cond-phi-2}
\end{align}
for every $(s,\xi), (s,\bar\xi) \in G$. By the mean value theorem
and~\eqref{cond-phi-1} we have
\begin{equation} \label{cond-phi-3}
   \|\phi(s,\xi) - \phi(s,\bar\xi)\| \le \|\xi-\bar\xi\|
\end{equation}
for every $(s,\xi), (s,\bar\xi) \in G$ and putting $\bar\xi= 0$
in~\eqref{cond-phi-3} we get
\begin{equation*}
   \|\phi(s,\xi)\| \le \|\xi\|
\end{equation*}
for every $(s,\xi) \in G$.

For each $\phi \in \cX$ we define the graph
\begin{equation} \label{def:V_phi}
   \cV_\phi
   = \set{\prts{s,\xi, \phi(s,\xi)} \colon (s,\xi) \in G}.
\end{equation}

Writing the unique solution of \eqref{eq:ivp-nonli} in the form
   $$ (x(t,s,v_s),y(t,s,v_s)) \in E(t) \times F(t),$$
where $v_s = (\xi, \eta) \in E(s) \times F(s)$, we define for each $\tau \ge 0$
the semiflow given by
\begin{equation} \label{def:Psi}
   \Psi_\tau(s,v_s) = \prts{s+\tau, x(s+\tau,s,v_s), y(s+\tau, s, v_s)}.
\end{equation}

Now we will formulate our theorem on the existence of global stable manifolds.

\begin{theorem}\label{thm:global}
   Let $X$ be a Banach space, assume that equation \eqref{eq:ivp-li} admits a
   nonuniform $\mu$-dichotomy in $\R^+_0$ for some $D \ge 1$, $a < 0 \le
   b$ and $\eps > 0$, and let $f \colon \R^+_0 \times X \to X$ be a function
   satisfying \eqref{cond-f-0},~\eqref{cond-f-1} and~\eqref{cond-f-2} for some
   $\delta > 0$. If
   \begin{equation} \label{eq:a+eps<b}
      a + \eps < b,
   \end{equation}
   then, choosing $\delta > 0$ sufficiently small, there exists a unique
   function $\phi \in \cX$ such that
   \begin{equation} \label{eq:invariance}
      \Psi_\tau(\cV_\phi) \subseteq \cV_\phi
   \end{equation}
   for every $\tau \ge 0$, where $\Psi_\tau$ is given by~\eqref{def:Psi} and
   $\cV_\phi$ is given by~\eqref{def:V_phi}. Furthermore,
   \begin{enumerate}[$a)$]
      \item $\cV_\phi$ is a $C^1$ manifold with $T_{(s,0)} \cV_\phi
         = \R \times E(s)$ for $s \ge 0$;
      \item there is $K > 0$ such that for every $(s, \xi), \ (s, \bar
          \xi) \in G$ and $t \ge s$ we have
          \begin{align}
            & \| \Psi_{t-s}(p_{s,\xi}) - \Psi_{t-s}(p_{s,\bar\xi})\|
               \le K \pfrac{\mu(t)}{\mu(s)}^a \mu(s)^\eps
               \|\xi - \bar\xi\| \label{eq:thm:global-1}\\
            & \| \partial \prts{\Psi_{t-s}(p_{s,\xi})}
               - \partial \prts{\Psi_{t-s}(p_{s,\bar\xi})}\|
               \le K \pfrac{\mu(t)}{\mu(s)}^a \mu(s)^{2\eps}
               \|\xi - \bar\xi\| \label{eq:thm:global-2}
          \end{align}
          where $p_{s,\xi} = (s, \xi, \phi(s,\xi))$.
   \end{enumerate}
\end{theorem}

Since problem \eqref{eq:ivp-nonli} is equivalent to the problem
\begin{align}
   & x(t) = U(t,s) \xi + \int_s^t U(t,r) f(r,x(r),y(r)) \dr,
      \label{eq:split-1a}\\
   & y(t) = V(t,s) \eta + \int_s^t V(t,r) f(r,x(r),y(r)) \dr,
   \label{eq:split-1b}
\end{align}
to prove the invariance in~\eqref{eq:invariance} we should have
\begin{align}
   & x(t,\xi) = U(t,s) \xi  + \int_s^t U(t,r)
      f(r,x(r,\xi),\phi(r,x(r,\xi))) \dr, \label{eq:split-2a}\\
   & \phi(t,x(t,\xi)) = V(t,s) \phi(s,\xi) + \int_s^t V(t,r)
      f(r,x(r,\xi),\phi(r,x(r,\xi))) \dr \label{eq:split-2b}
\end{align}
for every $s \ge 0$, every $t \ge s$ and every $\xi \in E(s)$.

The proof of the theorem goes as follows: in Lemma~\ref{lemma:global:aux1} we
prove, using Banach fixed point theorem in a suitable space $\cB_s$ of
functions, that for every $\phi \in \cX$, there is a unique function $x_\phi
\in \cB_s$ verifying~\eqref{eq:split-2a}; in Lemma~\ref{lemma:without-Gronwall}
we estimate the distance between two solutions $x_\phi$ and $x_\psi$ given by
Lemma~\ref{lemma:global:aux1}; then we establish in
Lemma~\ref{lemma:global:equiv} the equivalence between~\eqref{eq:split-2b} with
$x = x_\phi$ and a different equation; after that, another application of the
Banach fixed point theorem (this time in space $\cX$) gives us a unique
solution of the new equation and the theorem follows easily.

For $s \ge 0$ and $C > D$, we denote by $\cB=\cB_s$ the space of $C^1$
functions
   $$ x \colon [s,+\infty[ \times E(s) \to X$$
that, for every $t \ge s$ and $\xi, \bar\xi \in E(s)$, verify the following
conditions
\begin{align}
   & x(t,\xi) \in E(t) \label{cond-x-in-E(t)}\\
   & x(s, \xi) = \xi, \ x(t,0) = 0 \label{cond-x-0}\\
   & \|\partial x(t,\xi)\|
      \le C \pfrac{\mu(t)}{\mu(s)}^a \mu(s)^\eps \label{cond-x-1}\\
   & \|\partial x(t,\xi)- \partial x(t,\bar\xi)\|
      \le C \pfrac{\mu(t)}{\mu(s)}^a \mu(s)^{2\eps} \|\xi - \bar \xi\|
      \label{cond-x-2}
\end{align}

The mean value theorem and~\eqref{cond-x-1} imply that
\begin{equation} \label{cond-x-3}
   \|x(t,\xi) - x(t,\bar\xi)\|
   \le C \pfrac{\mu(t)}{\mu(s)}^a \mu(s)^\eps \|\xi - \bar\xi\|
\end{equation}
for every $t \ge s$ and $\xi,\bar\xi \in E(s)$ and when $\bar\xi= 0$ we have
the following estimate
\begin{equation} \label{cond-x-4}
   \|x(t,\xi)\| \le C \pfrac{\mu(t)}{\mu(s)}^a \mu(s)^\eps \|\xi\|
\end{equation}
for every $t \ge s$ and $\xi \in E(s)$. In space $\cB_s$ we consider the
metric induced by
\begin{equation} \label{norm:B}
   \|x\|'
   = \sup\set{\dfrac{\mu(s)^a }{\mu(s)^\eps \mu(t)^a} \
   \dfrac{\|x(t,\xi)\|}{\|\xi\|} \colon t \ge s, \ \xi \in E(s)
   \setminus\set{0}}.
\end{equation}

\begin{proposition} \label{prop:B:complete}
   The space $\cB_s$ is a complete metric space with the metric induced
   by~\eqref{norm:B}.
\end{proposition}

\begin{proof}
   For $x \in \cB_s$, $t \ge s$ and  $r>0$, we define a function $x^{t,r}
   \colon B_s(r) \to F(t)$ by
      $$ x^{t,r}(\xi) = x(t,\xi),$$
   where $B_s(r)$ is the open ball of $E(s)$ centered at $0$ and with radius
   $r$. Let $(x_n)_{n \in \N}$ be a Cauchy sequence in $\cB_s$ with respect to
   the metric induced by \eqref{norm:B}. Then the sequence
   $\prts{x_n^{t,r}}_{n\in \N}$ is a Cauchy sequence with respect the supremum
   norm in the space of bounded functions from $B_s(r)$ into $F(t)$. Hence,
   there is a function $x^{t,r} : B_s(r) \to F(t)$ such that
   $\prts{x^{t,r}_n}_{n \in \N}$ converges to $x^{t,r}$ in the space of bounded
   functions from $B_s(r)$ into $F(t)$ equipped with the supremum norm.

   For each $\xi, \bar\xi \in B_s(r)$, by~\eqref{cond-x-4},
   \eqref{cond-x-1} and \eqref{cond-x-2}, and denoting the first derivative by
   $D$, we obtain
   \begin{align*}
      & \|x^{t,r}_n(\xi)\| \le C \pfrac{\mu(t)}{\mu(s)}^a \mu(s)^\eps r \\
      & \|\prts{D x^{t,r}_n} (\xi) \|\le C \pfrac{\mu(t)}{\mu(s)}^a
      \mu(s)^\eps,\\
      & \| \prts{D x^{t,r}_n}(\xi) - \prts{D x^{t,r}_n}(\bar\xi)\|
      \le C \pfrac{\mu(t)}{\mu(s)}^a \mu(s)^{2\eps} \|\xi - \bar\xi\|.
   \end{align*}
   Denote by $C^{1,1}_b(B_s(r),F(t))$ the space of $C^1$ functions $u$, defined
   from $B_s(r)$ into $F(t)$, having Lipschitz derivative and such that
   $\|u\|_{1,1} \le b$, where $\|\cdot\|_{1,1}$ is defined by
      $$ \|u\|_{1,1} =\max \{\|u\|_\infty, \|Du\|_\infty, L(Du)\},$$
   $\lVert \cdot \rVert_\infty$ is the supremum norm and
      $$ L(u)= \sup \set{\frac{\|u(\xi)-u(\bar\xi)\|}{\|\xi-\bar\xi\|}
         : \xi, \bar\xi \in B_s(r) \text{ with } \xi \ne \bar\xi}.$$
   With
      $$ b
         = C \pfrac{\mu(t)}{\mu(s)}^a \mu(s)^\eps \max\set{\mu(s)^\eps, r},$$
   it follows that $x^{t,r}_n \in C^{1,1}_b(B_s(r),F(t))$.

   From the generalization of Henry's Lemma (see~\cite[p.151]{Henry-LNM-1981})
   given by Elbialy~\cite{Elbialy-PAMS-2000} (for related results see
   also~\cite{Lanford-III-LNM-1973} and~\cite{Chow-Lu-PRSE-1988}) we conclude
   that $x^{t,r} \in C_b^{1,1}(B_s(r),F(t))$ and
   \begin{equation}\label{gb_eq:tontet}
      \prts{D x^{t,r}_n}_{n \in \N} \text{ converges pointwise to } D
      x^{t,r} \text{ when } n \to \infty
   \end{equation}
   for every $\xi \in B_s(r)$. The uniqueness of each function $x^{t,r}$ in
   the ball~$B_s(r)$ implies that we can obtain a function $x \colon
   [s,+\infty[ \times E(s) \to X $ such that $x(t,\xi) = x^{t,r}(\xi)$ for
   each~$r > 0$, $t \ge s$ and $\xi \in B_s(r)$. From~\eqref{gb_eq:tontet} we
   can easily see that $x \in \cB_s$. Furthermore, because $(x_n)_{n \in \N}$
   is a Cauchy sequence, for each $\kappa >0$ there is $p \in \N$ such that for
   $n,m > p$ we have
   \begin{equation}\label{gb_eq:tontet.1}
      \|x_n(t,\xi)-x_m(t,\xi)\|
      \le \kappa \pfrac{\mu(t)}{\mu(s)}^a \mu(s)^\eps \|\xi\|
   \end{equation}
   for every $t \ge s$ and $\xi \in E(s)$. Letting $m \to \infty$
   in~\eqref{gb_eq:tontet.1} we get
      $$ \|x_n(t,\xi)-x(t,\xi)\|
         \le \kappa \pfrac{\mu(t)}{\mu(s)}^a \mu(s)^\eps \|\xi\|,$$
   and therefore $(x_n)_{n \in \N}$ converges to $x$ in the space~$\cB_s$.
\end{proof}

We equip the space $\cX$ with the metric induced by
\begin{equation}\label{norm:global:X}
   \|\phi\|'
   = \sup\set{\dfrac{\|\phi(s,\xi)\|}{\|\xi\|} \colon (s, \xi) \in G}.
\end{equation}

\begin{proposition} \label{prop:X:complete}
   The space $\cX$ is a complete metric space with the metric induced
   by~\eqref{norm:global:X}.
\end{proposition}

The proof of Proposition~\ref{prop:X:complete} is similar to the proof of
Proposition~\ref{prop:B:complete} and therefore is omitted.

Define
   $$ \phi_x(r,\xi) = \phi(r,x(r,\xi)) \quad \text{ and } \quad
      f_{x,\phi}(r,\xi) = f(t,x(r,\xi),\phi_x(r,\xi))$$
for each $x \in \cB_s$ and $\phi \in \cX$.

\begin{lemma} \label{lemma:global:aux1}
   Let $s \ge 0$ and $\phi \in \cX$. For $\delta > 0$ sufficiently
   small, there is one and only one $x = x_\phi \in \cB_s$ such that
   \begin{equation} \label{eq:x=U+int}
       x(t,\xi)
       = U(t,s) \xi + \int_s^t U(t,r) f_{x,\phi}(r,\xi) \dr
   \end{equation}
   for every $t \ge s$ and $\xi \in E(s)$.
\end{lemma}

\begin{proof}
   Given $\phi \in \cX$, consider in $\cB_s$ the operator $J=J_\phi$ given,
   for every $x \in \cB_s$, by
      $$ (J x)(t,\xi)
         = U(t,s) \xi + \int_s^t U(t,r) f_{x,\phi}(r,\xi) \dr$$
   for each $(t,\xi) \in [s,+\infty[ \times E(s)$. First we will prove that $Jx
   \in \cB_s$ for every $x \in \cB_s$.

   The definition of $J$ immediately assures that $(Jx)(s,\xi) = U(s,s)
   \xi = \xi$ for every $\xi \in E(s)$ and that $(Jx)(t, \xi) \in E(t)$ for
   every $t \ge  s$ and every $\xi \in E(s)$. Furthermore,
   from~\eqref{cond-x-0},~\eqref{cond-phi-0} and~\eqref{cond-f-0} we obtain
   $(Jx)(t,0) = 0$ for every $t \ge s$.

   Moreover, the operator $J$ is of class $C^1$ and
      $$ \partial (Jx) (t,\xi)
         = U(t,s) + \int_s^t U(t,r) \ \partial f_{x,\phi}(r,\xi) \dr$$
   and this implies that
   \begin{equation}\label{eq:norm:der:Jx}
      \|\partial(Jx) (t,\xi)\|
      \le \|U(t,s)\| + \int_s^t \|U(t,r)\| \|\partial f_{x,\phi}(r, \xi)\| \dr.
      \end{equation}
   From the chain rule and~\eqref{cond-phi-1} it follows that
   \begin{align*}
      \| \partial f_{x,\phi}(r,\xi)\|
      & \le \|\partial f(r, x(r,\xi), \phi(r, x(r, \xi)))\|
         \prtsr{\|\partial x(r,\xi)\| + \|\partial \phi(r,x(r,\xi))\|
         \|\partial x(r,\xi)\|}\\
      & \le 2 \|\partial f(r, x(r,\xi), \phi(r, x(r, \xi)))\| \
         \|\partial x(r,\xi)\|
   \end{align*}
   for every $r \ge s$ and every $\xi \in E(s)$. By~\eqref{cond-f-1}
   and~\eqref{cond-x-1} we obtain
   \begin{equation} \label{eq:norm:der:f_phi,x}
      \| \partial f_{x,\phi}(r,\xi)\|
         \le 2 C \delta \mu'(r)\mu(r)^{-3 \eps - 1}
         \pfrac{\mu(r)}{\mu(s)}^a \mu(s)^\eps.
   \end{equation}
   for every $r \ge s$ and every $\xi \in E(s)$. Using~\eqref{eq:dich-1}
   and~\eqref{eq:norm:der:f_phi,x} we get
   \begin{align*}
      \int_s^t \|U(t,r)\| \|\partial f_{x,\phi}(r, \xi)\| \dr
      & \le 2CD\delta \pfrac{\mu(t)}{\mu(s)}^a \mu(s)^\eps
         \int_s^t \mu'(r) \mu(r)^{-2\eps -1 }\dr\\
      & \le \dfrac{CD \delta}{\eps}
         \pfrac{\mu(t)}{\mu(s)}^a \mu(s)^{-\eps}.
   \end{align*}
   Using this inequality,~\eqref{eq:dich-1} and~\eqref{eq:norm:der:Jx} we have
      $$ \| \partial(Jx) (t,\xi)\|
         \le \prts{D + \dfrac{CD \delta}{\eps}}
            \pfrac{\mu(t)}{\mu(s)}^a \mu(s)^\eps$$
   for every $t \ge s$ and every $\xi \in E(s)$. Choosing $\delta \le \eps
   \prts{\dfrac{1}{D}-\dfrac{1}{C}}$ we obtain for every $t \ge s$ and every
   $\xi \in E(s)$
      $$ \| \partial(Jx) (t,\xi)\|
         \le C \pfrac{\mu(t)}{\mu(s)}^a \mu(s)^\eps.$$

   For $r \ge s$ and $\xi \in E(s)$, using again the chain
   rule,~\eqref{cond-f-1},~\eqref{cond-f-2},~\eqref{cond-phi-1},~\eqref{cond-phi-2}
   and~\eqref{cond-phi-3}, we have
   \begin{align*}
      & \norm{\partial f_{x,\phi}(r, \xi) - \partial f_{x,\phi}(r, \bar\xi)}\\
      & \le \norm{\partial f(r, x(r,\xi), \phi_x(r, \xi))
         - \partial f(r, x(r,\bar\xi), \phi_x(r,\bar\xi))} \times\\
      & \phantom{ \le } \times
         \prts{\norm{\partial x(r,\xi)}
         + \norm{\partial \phi(r,x(r,\xi))} \norm{\partial x(r,\xi)}}\\
      & \phantom{ \le }
         + \norm{\partial f(r, x(r,\bar\xi), \phi_x(r,\bar\xi))}
         \norm{\partial x(r,\xi) - \partial x(r,\bar\xi)}\\
      & \phantom{ \le } +
         \norm{\partial f(r, x(r,\bar\xi), \phi_x(r,\bar\xi))}
         \norm{\partial \phi(r,x(r,\xi))} \norm{\partial x(r,\xi)
            - \partial x(r,\bar\xi)}\\
      & \phantom{ \le } +
         \norm{\partial f(r, x(r,\bar\xi), \phi_x(r,\bar\xi))}
         \norm{\partial \phi(r,x(r,\xi)) - \partial \phi(r,x(r,\bar\xi))}
         \norm{\partial x(r,\bar\xi)}\\
      & \le 2 \delta \mu'(r) \ \mu(r)^{-3\eps-1} \prts{\norm{x(r,\xi) -
      x(r,\bar\xi)}
         + \norm{ \phi_x(r, \xi) -  \phi_x(r,\bar\xi)}}
         \norm{\partial x(r,\xi)}\\
      & \phantom{ \le }
         + 2 \delta \mu'(r) \ \mu(r)^{-3\eps-1}
         \norm{\partial x(r,\xi) - \partial x(r,\bar\xi)}\\
      & \phantom{ \le } +
         \delta \mu'(r) \ \mu(r)^{-3\eps-1}
         \norm{x(r,\xi) - x(r,\bar\xi)} \norm{\partial x(r,\bar\xi)}\\
      & \le \delta \mu'(r) \ \mu(r)^{-3\eps-1}
         \prts{2 \norm{\partial x(r,\xi) - \partial x(r,\bar\xi)}
         + 5 \norm{x(r,\xi) - x(r,\bar\xi)} \norm{\partial x(r,\bar\xi)}}
   \end{align*}
   and by~\eqref{cond-x-2},~\eqref{cond-x-3} and~\eqref{cond-x-1} we obtain
   \begin{equation} \label{eq:norm:der:f-f}
      \norm{\partial f_{x,\phi}(r, \xi) - \partial f_{x,\phi}(r, \bar\xi)}
      \le 7 C^2 \delta \mu'(r) \ \mu(r)^{-3\eps-1}
         \pfrac{\mu(r)}{\mu(s)}^a \mu(s)^{2\eps} \|\xi-\bar\xi\|.
   \end{equation}
   Therefore, for every $t \ge s$ and every $\xi \in E(s)$, we obtain
   \begin{align*}
      \|\partial(J x)(t,\xi) - \partial(Jx)(t,\bar\xi)\|
      & \le \int_s^t \|U(t,r)\|
         \|\partial f_{x,\phi}(r, \xi) - \partial f_{x,\phi}(r,\bar\xi)\| \dr
         \\
      & \le 7C^2 D \delta \pfrac{\mu(t)}{\mu(s)}^a \mu(s)^{2 \eps}
         \|\xi-\bar\xi\| \int_s^t  \mu'(r) \ \mu(r)^{-2\eps - 1} \dr\\
      & \le \dfrac{7 C^2 D \delta}{2 \eps}  \pfrac{\mu(t)}{\mu(s)}^a
         \|\xi-\bar\xi\|
   \end{align*}
   and for $\delta < \dfrac{2 \eps}{7 C D}$ we have
      $$ \|\partial(J x)(t,\xi) - \partial(Jx)(t,\bar\xi)\|
         \le C  \pfrac{\mu(t)}{\mu(s)}^a \mu(s)^{2\eps}
         \|\xi-\bar\xi\|.$$

   We have proved $Jx$
   verifies~\eqref{cond-x-in-E(t)},~\eqref{cond-x-0},~\eqref{cond-x-1}
   and~\eqref{cond-x-2} for every $x \in \cB_s$. Therefore $J$ is an operator
   from $\cB_s$ into $\cB_s$.

   Now we will prove that, choosing $\delta$ sufficiently small, $J$ is
   a contraction in $\cB_s$. From~\eqref{cond-f-3} and~\eqref{cond-phi-3} we
   have for $r \ge s$ and $\xi \in E(s)$
   \begin{align*}
      \|f_{x,\phi}(r,\xi) - f_{y,\phi}(r,\xi)\|
      & \le \delta \mu'(r) \mu(r)^{-3 \eps -1}
         \| \prts{x(r,\xi),\phi_x(r,\xi)} - \prts{y(r,\xi),\phi_y(r,\xi)} \|\\
      & \le 2 \delta \mu'(r) \mu(r)^{-3 \eps -1} \| x(r,\xi) - y(r,\xi)\| \\
      & \le 2 \delta \mu'(r) \mu(r)^{-3 \eps -1}
         \pfrac{\mu(r)}{\mu(s)}^a \mu(s)^\eps  \|\xi\| \| x - y\|'.
   \end{align*}
   Using this estimate, we obtain by~\eqref{eq:dich-1}, for every $t \ge s$ and
   every $\xi \in E(s)$,
   \begin{align*}
      \|(Jx)(t,\xi) - (Jy)(t,\xi)\|
      & \le \int_s^t \|U(t,r)\| \|f_{x,\phi}(r,\xi)-f_{y,\phi}(r,\xi)\| \dr\\
      & \le 2 D \delta \pfrac{\mu(t)}{\mu(s)}^a \mu(s)^\eps \|\xi\|
         \| x - y\|' \int_s^t \mu'(r) \mu(r)^{-2 \eps -1} \dr\\
      & \le \dfrac{D \delta}{\eps}
         \pfrac{\mu(t)}{\mu(s)}^a \mu(s)^\eps \|\xi\| \| x - y\|'
   \end{align*}
   and thus
      $$ \|Jx - Jy\|' \le \dfrac{D \delta}{\eps} \| x - y\|' $$
   for every $x,y \in \cB_s$. Therefore, choosing $\delta < \eps/D$, we
   conclude that $J$ is a contraction. Because $\cX$ is a complete metric
   space, $J$ has a unique fixed point $x_\phi \in \cB_s$ and this fixed point
   verifies~\eqref{eq:x=U+int}. This concludes the proof.
\end{proof}

Given $\phi \in \cX$ we denote by $x_\phi$ the unique function in $\cB_s$ that
verifies~\eqref{eq:x=U+int}. In the next Lemma we obtain an estimate that, in
the exponential case, is usually obtained using Gronwall's lemma. Here we use
an induction argument that allows us to obtain a corresponding estimate in our
generalized context.

\begin{lemma} \label{lemma:without-Gronwall}
   Choosing $\delta >0$ sufficiently small, we have
   \begin{equation} \label{eq:norm_x_phi-x_psi}
      \|x_\phi(t,\xi) - x_\psi(t,\xi)\|
      \le C \pfrac{\mu(t)}{\mu(s)}^a \mu(s)^{-\eps} \|\xi\| \cdot \|\phi -
      \psi\|'
   \end{equation}
   for every $\phi, \psi \in \cX$ and every $(t,\xi) \in [s,+\infty[ \times
   E(s)$.
\end{lemma}

\begin{proof}
   Given $\phi, \psi \in \cX$, we write $y_{n+1} = J_\phi y_n$ and $z_{n+1} =
   J_\psi z_n$ with
      $$ y_1(t,\xi) = z_1(t,\xi) = U(t,s)\xi$$
   for every $t \ge s$ and every $\xi \in E(s)$. Since $x_\phi$ and $x_\psi$
   were obtained in Lemma~\ref{lemma:global:aux1} using Banach's fixed point
   theorem, it follows that
      $$ \|x_\phi - x_\psi\|'
         = \lim_{n \to +\infty} \| y_n - z_n\|'.$$
   Hence, to prove \eqref{eq:norm_x_phi-x_psi}, it is enough to prove that, for
   each $n \in \N$, we have
   \begin{equation} \label{eq:J^n_phi-J^n_psi}
      \| y_n(t,\xi) - z_n(t,\xi)\|
      \le C \pfrac{\mu(t)}{\mu(s)}^a \mu(s)^{-\eps} \|\xi\| \cdot \|\phi -
      \psi\|'
   \end{equation}
   for all $(t,\xi) \in [s,+\infty[ \times E(s)$. We are going to prove
   inequality~\eqref{eq:J^n_phi-J^n_psi} by mathematical induction on $n$.
   Obviously, inequality~\eqref{eq:J^n_phi-J^n_psi} holds for $n=1$. Suppose
   that \eqref{eq:J^n_phi-J^n_psi} holds for $n$. Then, for $r \ge s$ and $\xi
   \in E(s)$, we have by~\eqref{cond-f-2},~\eqref{cond-x-4} and the induction
   hypothesis
   \begin{align*}
      & \|f_{y_n,\phi}(r,\xi) - f_{z_n,\psi}(r,\xi)\|\\
      & \le \delta \mu'(r) \ \mu(r)^{-3\eps-1}
         \prts{\|y_n(r,\xi)-z_n(r,\xi)\|
            + \|\phi_{y_n}(r,\xi)- \psi_{z_n}(r,\xi)\|}\\
      & \le \delta \mu'(r) \ \mu(r)^{-3\eps-1}
         \prts{2 \|y_n(r,\xi)-z_n(r,\xi)\|
            + \|\phi_{z_n}(r,\xi)- \psi_{z_n}(r,\xi)\|}\\
      & \le \delta \mu'(r) \ \mu(r)^{-3\eps-1}
         \prts{2 \|y_n(r,\xi)-z_n(r,\xi)\|
            + \|\phi-\psi\|' \|z_n(r,\xi)\|}\\
      & \le 3 \delta C \mu'(r) \ \mu(r)^{-3\eps-1} \pfrac{\mu(r)}{\mu(s)}^a
      \mu(s)^\eps
         \|\xi\| \cdot \|\phi - \psi\|'
   \end{align*}
   and this implies that
   \begin{align*}
      & \| y_{n+1}(t,\xi) - z_{n+1}(t,\xi)\| \\
      & \le \int_s^t \|U(t,r)\|
         \|f_{y_n,\phi}(r,\xi) - f_{z_n,\psi}(r,\xi)\| \dr\\
      & \le 3 C D \delta  \pfrac{\mu(t)}{\mu(s)}^a \mu(s)^\eps
         \|\xi\| \cdot \|\phi - \psi\|'
         \int_s^t \mu'(r) \mu(r)^{-2\eps-1}\dr\\
      & \le \dfrac{3 C D \delta}{2 \eps} \pfrac{\mu(t)}{\mu(s)}^a
         \mu(s)^{-\eps} \|\xi\| \cdot \|\phi - \psi\|'
   \end{align*}
   for every $t \ge s$ and every $\xi \in E(s)$. Thus, choosing $\delta <
   \dfrac{2 \eps}{3 D}$, inequality~\eqref{eq:J^n_phi-J^n_psi} holds for
   $n+1$.
   Therefore \eqref{eq:J^n_phi-J^n_psi} is true for all $n \in \N$ and this
   finishes the proof of the lemma.
\end{proof}

\begin{lemma} \label{lemma:global:equiv}
   If $\delta > 0$ is sufficiently small, for every $\phi \in \cX$, the
   following properties are equivalent
   \begin{enumerate}[$a)$]
      \item for every $s\ge 0$, $t \ge s$ and $\xi \in E(s)$,
         \begin{equation} \label{eq:phi_x}
            \phi_{x_\phi} (t,\xi)
            = V(t,s) \phi(s,\xi)
               + \int_s^t V(t,r) f_{x_\phi, \phi} (r,\xi) \dr,
         \end{equation}
         where $x_\phi \in \cB_s$ is given by
         Lemma~\ref{lemma:global:aux1};
      \item for every $s \ge 0$ and every $\xi \in E(s)$,
         \begin{equation} \label{eq:phi}
            \phi(s,\xi)
            = -\int_s^{+\infty} V(r,s)^{-1} f_{x_\phi, \phi} (r,\xi) \dr,
         \end{equation}
         where $x_\phi \in \cB_s$ is given by
         Lemma~\ref{lemma:global:aux1}.
   \end{enumerate}
\end{lemma}

\begin{proof}
   We start by showing that the integral in~\eqref{eq:phi} is well defined. For
   every $r \ge s$, from~\eqref{cond-f-4} and using the fact that $x_\phi \in
   \cB_s$, we obtain
   \begin{align*}
      \|f_{x_\phi,\phi}(r,\xi)\|
      & \le 2 \delta \mu'(r) \ \mu(r)^{-3\eps-1} \|x_\phi(r,\xi)\|\\
      & \le 2 C \delta \mu'(r) \ \mu(r)^{-3\eps-1} \pfrac{\mu(r)}{\mu(s)}^a
            \mu(s)^\eps \|\xi\|
   \end{align*}
   and by~\eqref{eq:dich-2} this implies that
   \begin{align*}
      & \int_s^{+\infty} \|V(r,s)^{-1}\| \|f_{x_\phi, \phi} (r,\xi)\| \dr\\
      & \le 2CD\delta \mu(s)^{-a + b + \eps} \|\xi\|
         \int_s^{+\infty} \mu'(r) \mu(r)^{a-b-2\eps-1} \dr\\
      & \le \dfrac{2CD \delta}{\abs{a -b - 2\eps}} \mu(s)^{-\eps} \|\xi\|.
   \end{align*}
   Therefore the integral in~\eqref{eq:phi} is well defined.

   Now let us assume that $a)$ is verified. Then, from~\eqref{eq:phi_x} we get
   \begin{equation} \label{eq:phi-aux1}
      \phi(s,\xi)
      = V(t,s)^{-1} \phi_{x_\phi} (t,\xi)- \int_s^t V(r,s)^{-1}
         f_{x_\phi, \phi} (r,\xi) \dr.
   \end{equation}
   Since
   \begin{align*}
      \|V(t,s)^{-1} \phi_{x_\phi} (t,\xi)\|
      & \le D \pfrac{\mu(t)}{\mu(s)}^{-b} \mu(t)^\eps \|x_\phi(t,\xi)\|\\
      & \le D \pfrac{\mu(t)}{\mu(s)}^{-b} \mu(t)^\eps
         C \pfrac{\mu(t)}{\mu(s)}^a \mu(s)^\eps \|\xi\|\\
      & = DC \|\xi\| \mu(s)^{-a+b+\eps} \mu(t)^{a-b+\eps},
   \end{align*}
   we can conclude from~\eqref{eq:a+eps<b} that $V(t,s)^{-1} \phi_{x_\phi}
   (t,\xi)$ converges to zero as $t$ converges to $+ \infty$. Therefore we
   obtain $b)$ letting $t \to + \infty$ in~\eqref{eq:phi-aux1}.

   Assume now that $b)$ holds. Defining a semiflow $F_r$ on $G$ by
      $$ F_r(s,\xi)=(s+r, x_\phi(s+r,\xi)),$$
   we obtain from~\eqref{eq:phi}
   \begin{equation} \label{eq:phi-aux2}
      \phi(s,\xi)
      = - \int_s^{+\infty} V(r,s)^{-1}
         f(F_{r-s}(s,\xi),\phi(F_{r-s}(s,\xi)))\dr.
   \end{equation}
   Replacing $(s, \xi)$ by $(t,x_\phi(t,\xi))$ in~\eqref{eq:phi-aux2} we get
   \begin{align*}
      \phi(t,x_\phi(t,\xi))
      & = - \int_t^{+\infty} V(r,t)^{-1}
         f(F_{r-t}(t,x_\phi(t,\xi)),\phi(F_{r-t}(t,x_\phi(t,\xi))))\dr\\
      & = - \int_t^{+\infty} V(r,t)^{-1}
         f(r,x_\phi(r,\xi),\phi(r,x_\phi(t,\xi)))\dr,
   \end{align*}
   because $F_r$ is a semiflow and this implies
      $$ F_{r-t}(t,x_\phi(t,\xi))
         = F_{r-t}(F_{t-s}(s,\xi))
         = F_{r-s}(s,\xi)
         = (r,x_\phi(r,\xi)).$$

   Then, since $V(t,s) V(r,s)^{-1} = V(t,r)$, we obtain
   \begin{align*}
      V(t,s) \phi(s,\xi)
      & = - \int_s^{+\infty} V(t,s) V(t,r)^{-1} f_{x_\phi,\phi} (r,\xi) \dr\\
      & = - \int_s^t V(t,r) f_{x_\phi,\phi} (r,\xi) \dr
         - \int_t^{+\infty} V(t,r) f_{x_\phi,\phi} (r,\xi) \dr\\
      & = - \int_s^t V(t,r) f_{x_\phi,\phi} (r,\xi) \dr
         + \phi(t,x_\phi(t,\xi))
   \end{align*}
   and thus
      $$ V(t,s) \phi(s,\xi) + \int_s^t V(t,r) f_{x_\phi,\phi}(r,\xi)\dr
         = \phi(t,x_\phi(t,\xi)).$$
  Hence $b)$ implies $a)$.
\end{proof}

Next we will prove, for $\delta$ sufficiently small, the existence of a unique
function $\phi \in \cX$ that verifies~\eqref{eq:phi}.

\begin{lemma} \label{lemma:global:aux2}
   Choosing $\delta > 0$ sufficiently small, there is a unique $\phi \in \cX$
   such that~\eqref{eq:phi} holds for every $s \ge 0$ and every $\xi \in
   E(s)$.
\end{lemma}

\begin{proof}
   Let $\Phi$ be an operator on $\cX$ defined by
   \begin{equation*}
      \prts{\Phi \phi}(s,\xi)
         = - \int_s^{+\infty} V(r,s)^{-1} f_{x_\phi, \phi} (r,\xi) \dr
   \end{equation*}
   for each $\phi \in \cX$ and each $(s,\xi) \in G$. First we will prove that
   $\Phi$ is an operator from $\cX$ into $\cX$.

   Obviously, $(\Phi\phi)(s,\xi) \in F(s)$ for every $(s,\xi) \in G$.
   Furthermore, $\prts{\Phi \phi}(s,0) = 0$ for every $s \ge 0$ because
   $x_\phi(r,0) = 0$ for every $\phi \in \cX$ and every $r \ge s$. Moreover,
   $\Phi \phi$ is of class $C^1$ and
   \begin{equation*}
      \partial (\Phi \phi)(s,\xi)
         = - \int_s^{+\infty} V(r,s)^{-1}
            \partial f_{x_\phi, \phi} (r,\xi) \dr.
   \end{equation*}
   By~\eqref{cond-f-0} we have $\partial (\Phi \phi)(s,0) = 0$ for every $s
   \ge
   0$.

   From~\eqref{eq:dich-2} and~\eqref{eq:norm:der:f_phi,x} we have
   \begin{align*}
      \|\partial (\Phi \phi)(s,\xi)\|
      & \le \int_s^{+\infty} \|V(r,s)^{-1}\| \cdot
            \|\partial f_{x_\phi, \phi} (r,\xi)\| \dr\\
  \dr\\
      & \le 2 C D \delta \mu(s)^{-a+b + \eps}
         \int_s^{+\infty} \mu'(r) \mu(r)^{a-b-2\eps -1} \dr\\
      & \le \dfrac{2 C D \delta}{|a - b - 2 \eps|}
   \end{align*}
   and choosing $\delta < \dfrac{|a - b - 2 \eps|}{2 C D}$ we obtain for every
   $s \ge 0$ and every $\xi \in E(s)$
      $$ \|\partial (\Phi \phi)(s,\xi)\| \le 1.$$

   It follows from~\eqref{eq:dich-2} and~\eqref{eq:norm:der:f-f} that
   \begin{align*}
      & \|\partial (\Phi \phi)(s,\xi) - \partial (\Phi \phi)(s,\bar\xi)\| \\
      & \le \int_s^{+\infty} \|V(r,s)^{-1}\| \cdot
            \|\partial f_{x_\phi, \phi} (r,\xi)
               - \partial f_{x_\phi, \phi} (r,\bar\xi)\| \dr\\
  \mu(s)^{2\eps}
      & \le 7 C^2 D \delta \mu(s)^{-a+b+2\eps}\|\xi-\bar\xi\|
         \int_s^{+\infty} \mu'(r) \mu(r)^{a-b-2\eps -1} \dr \\
      & = \dfrac{7 C^2 D \delta}{|a -b -2\eps|} \|\xi-\bar\xi\|.
   \end{align*}
   Therefore, if $\delta \le \dfrac{|a -b  -2 \eps|}{7 C^2 D}$, we get
      $$ \|\partial (\Phi \phi)(s,\xi) - \partial (\Phi \phi)(s,\bar\xi)\|
         \le \|\xi-\bar\xi\|$$
   for every $s \ge 0$ and every $\xi,\bar\xi \in E(s)$.

   Hence, $\Phi \phi$
   satisfies~\eqref{cond-phi-in-F(s)},~\eqref{cond-phi-0},~\eqref{cond-phi-1}
   and~\eqref{cond-phi-2} for every $\phi \in \cX$ and this proves that
   $\Phi$ is an operator from $\cX$ into $\cX$.

   To finish the proof we will verify that $\Phi$ is a contraction for $\delta$
   sufficiently small. Let $\phi, \psi \in \cX$ and $(s,\xi) \in \R^+_0 \times
   E(s)$. By~\eqref{cond-f-3},~\eqref{cond-phi-3},~\eqref{eq:norm_x_phi-x_psi}
   and~\eqref{cond-x-4}, for $r \ge s$, we have
   \begin{align*}
      & \|f_{x_\phi,\phi}(r,\xi) - f_{x_\psi,\psi}(r,\xi)\|\\
      & \le \delta \mu'(r) \ \mu(r)^{-3\eps-1} \prts{\|x_\phi(r,\xi) -
      x_\psi(r,\xi)\|
         + \|\phi_{x_\phi}(r,\xi) - \psi_{x_\psi}(r,\xi)\|}\\
      & \le \delta \mu'(r) \ \mu(r)^{-3\eps-1} \prts{2 \|x_\phi(r,\xi) -
      x_\psi(r,\xi)\|
         + \|\phi_{x_\phi}(r,\xi) - \psi_{x_\phi}(r,\xi)\|}\\
      & \le \delta \mu'(r) \ \mu(r)^{-3\eps-1} \prts{2 \|x_\phi(r,\xi) -
      x_\psi(r,\xi)\|
         + \|\phi- \psi\|' \|x_\phi(r,\xi) \|}\\
      & \le 3 C \delta \mu'(r) \ \mu(r)^{-3\eps-1} \pfrac{\mu(r)}{\mu(s)}^a
         \mu(s)^{\eps} \|\xi\| \cdot \|\phi - \psi\|'.
   \end{align*}
   The last inequality and~\eqref{eq:dich-2} implies that
   \begin{align*}
      & \|(\Phi \phi)(s,\xi) - \Phi \psi (s,\xi)\|\\
      & \le \int_s^{+\infty} \|V(r,s)^{-1}\| \cdot
         \|f_{x_\phi,\phi}(r,\xi) - f_{x_\psi,\psi}(r,\xi)\| \dr\\
      & \le 3 C D \delta \, \|\xi\| \cdot \|\phi - \psi\|' \mu(s)^{-a+b+\eps}
         \int_s^{+\infty} \mu'(r) \mu(r)^{a - b - 2 \eps -1} \dr\\
      & \le \dfrac{3 C D \delta }{|a-b-2\eps|} \, \|\xi\|
         \cdot \|\phi - \psi\|'
   \end{align*}
    and this implies that
      $$ \| \Phi \phi - \Phi \psi\|'
         \le \dfrac{3 C D \delta }{|a-b-2\eps|} \, \|\phi - \psi\|'.$$
    Therefore, if $\delta < \dfrac{|a-b-2\eps|}{3 C D}$, we conclude that
    $\Phi$ is a contraction in $\cX$ and that the unique fixed point of $\Phi$
    verifies~\eqref{eq:phi}.
\end{proof}

\begin{proof}[Proof of Theorem \ref{thm:global}]
   By Lemma~\ref{lemma:global:aux1}, for each $\phi\in \cX$ there is
   a unique sequence $x_\phi \in \cB_s$ satisfying
   identity~\eqref{eq:split-2a}. By Lemma~\eqref{lemma:global:equiv}, solving
   equation~\eqref{eq:split-2b} with $x=x_\phi$ is equivalent to solve
   equation~\eqref{eq:phi}. Finally, by Lemma~\ref{lemma:global:aux2},
   there is a unique solution of \eqref{eq:phi}. Therefore, we obtain a unique
   solution of equation~\eqref{eq:split-2b} with $x = x_\phi$ for $\delta$
   sufficiently small.

   To prove that, for the function $\phi$ that solves~\eqref{eq:split-2b} with
   $x = x_\phi$, the graph $\cV_\phi$ is a $C^1$ manifold we have to consider
   the map
      $$ S \colon \R^+_0 \times E(0) \to \R^+_0 \times X$$
   defined by
      $$ S(t,\xi) = \Psi_t(0,\xi,\phi(0,\xi)).$$
   The map $S$ is of class $C^1$ because $\phi(0,\xi)$ is also of class $C^1$.
   Moreover, if $S(t,\xi) = S(t',\xi')$, then $t = t'$ and $\xi = \xi'$. Thus,
   $S$ is a parametrization of class $C^1$ of the set
   $\cV_\phi$. Therefore, $\cV_\phi$ is a $C^1$ manifold.

   Finally, for every $(s, \xi), (s, \bar \xi) \in G$ and every $t \ge s$,
   from~\eqref{cond-phi-3} we have
   \begin{align*}
      \| \Psi_{t-s}(p_{s,\xi}) - \Psi_{t-s} (p_{s,\bar\xi})\|
      & = \|\prts{t, x_\phi(t,\xi), \phi_{x_\phi} (t,\xi)}
         - \prts{t, x_\phi(t,\bar\xi), \phi_{x_\phi} (t,\bar\xi)}\|\\
      & \le \|x_\phi(t,\xi) - x_\phi(t,\bar\xi)\|
         + \|\phi_{x_\phi}(t,\xi) - \phi_{x_\phi}(t,\bar\xi)\| \\
      & \le 2 \|x_\phi(t,\xi) - x_\phi(t,\bar\xi)\| \\
      & \le 2 C \pfrac{\mu(t)}{\mu(s)}^a \mu(s)^\eps
               \|\xi - \bar\xi\|
   \end{align*}
   and this proves~\eqref{eq:thm:global-1}. To prove~\eqref{eq:thm:global-2},
   we note that
   \begin{align*}
      \|\partial \phi_{x_\phi}(t,\xi) - \partial \phi_{x_\phi}(t,\bar\xi)\|
      & \le \|\partial \phi(t,x_\phi(t,\xi)) - \partial
      \phi(t,x_\phi(t,\bar\xi))\|
         \cdot \|\partial x_\phi(t,\xi)\|\\
      & \phantom{ \le } + \|\partial \phi(t,x_\phi(t,\bar\xi))\|
         \cdot \|\partial x_\phi(t,\xi) - \partial x_\phi(t,\bar\xi)\|\\
      & \le \|x_\phi(t,\xi) - x_\phi(t,\bar\xi)\|
         \cdot \|\partial x_\phi(t,\xi)\|\\
      & \phantom{ \le } +
         \|\partial x_\phi(t,\xi) - \partial x_\phi(t,\bar\xi)\|\\
      & \le (C + C^2) \pfrac{\mu(t)}{\mu(s)}^a \mu(s)^{2\eps} \|\xi -
      \bar\xi\|.
   \end{align*}
   This last estimate is a consequence of the chain
   rule,~\eqref{cond-phi-2},~\eqref{cond-phi-1},~\eqref{cond-x-3},~\eqref{cond-x-1}
   and~\eqref{cond-x-2}. Therefore, by the last estimate and~\eqref{cond-x-2},
   we
   have
   \begin{align*}
      \| \partial\Psi_{t-s}(p_{s,\xi}) - \partial\Psi_{t-s} (p_{s,\bar\xi})\|
      & = \|\partial\prts{t, x_\phi(t,\xi), \phi_{x_\phi} (t,\xi)}
         - \partial \prts{t, x_\phi(t,\bar\xi), \phi_{x_\phi} (t,\bar\xi)}\|\\
      & \le \|\partial x_\phi(t,\xi) - \partial x_\phi(t,\bar\xi)\|
         + \|\partial \phi_{x_\phi}(t,\xi) - \partial
         \phi_{x_\phi}(t,\bar\xi)\| \\
      & \le (2C + C^2) \pfrac{\mu(t)}{\mu(s)}^a \mu(s)^{2\eps} \|\xi -
      \bar\xi\|.
   \end{align*}
   This completes the proof of the theorem.
\end{proof}

\bibliographystyle{elsart-num-sort}

\end{document}